\documentclass[11pt]{amsart}
\usepackage{amsmath, amssymb, amsthm}
\usepackage{geometry}
\usepackage{enumitem}
\usepackage[utf8]{inputenc}
\usepackage{hyperref}
\usepackage{newunicodechar}
\newunicodechar{−}{-}
\usepackage{amssymb}
\usepackage{enumitem}
\usepackage[dvipsnames]{xcolor}
\usepackage[normalem]{ulem}

\newunicodechar{，}{,}
\newcommand{\EQN}[1]{\begin{align*}#1\end{align*}}

\newtheorem{theorem}{Theorem}
\newtheorem{lemma}[theorem]{Lemma}
\newtheorem{proposition}[theorem]{Proposition}
\newtheorem{definition}[theorem]{Definition}
\newtheorem{remark}[theorem]{Remark}
\usepackage{color}
\newcommand{\R}{\mathbb{R}}

\newcommand{\Z}{\mathbb{Z}}

\newcommand{\PP}{\mathbb{P}}

\newcommand{\norm}[1]{\left\| #1 \right\|}

\begin{document}

\title[rotating incompressible Navier-Stokes equations]{Well-posedness and vanishing rotational limit for the rotating incompressible Navier-Stokes equations in hybird Besov space}

\author[Z.~Guo]{Zihua Guo}

\address{Zihua Guo, School of Matheamtics, Monash University, Melbourne, VIC 3800, Australia.}
\email{zihua.guo@monash.edu}

\author[Z. Song]{Zihao Song}
\address{Zihao SOng, School of Mathematics, Nanjing University of Aeronautics and Astronautics, Nanjing 211106, Jiangsu, China}
\email{szh1995@nuaa.edu.cn}

\author[M. Yang]{Minghua Yang}
\address{Minghua Yang, School of Information Management and Mathematics, Jiangxi University of Finance and
Economics, Nanchang, 330032, China}
\email{minghuayang@jxufe.edu.cn}

\thispagestyle{empty}
\begin{abstract}
We establish the well-posedness of the 3D rotating incompressible Navier-Stokes equations
with critical initial data  $u_{0,\Omega}\in X_{0,q,p}^{\Omega}$ for $p<5$, where $X_{0,q,p}^{\Omega}$ is defined by the norm
\begin{equation*}
\begin{aligned}
&\|u_{0,\Omega}\|_{X_{0,q,p}^{\Omega}}:= \Omega^{3- \frac{6}{q}}\|u_{0,\Omega}\|_{\dot{B}_{q,\infty}^{-7+\frac{15}{q}}}^{\ell_\Omega}
+\|u_{0,\Omega}\|_{\dot{B}_{p,\infty}^{-1+\frac{3}{p}}}^{h_\Omega}.
\end{aligned}
\end{equation*}
This extends the previous results by Chen, Miao, and Zhang (\cite{CMZ2013}).
The main ingredients are a new global-in-time dissipative-dispersive estimate for the Stokes--Coriolis semigroup and corresponding bilinear estimates. Furthermore, we establish the vanishing rotational limit for the 3D rotating Navier-Stokes equations as $\Omega\rightarrow 0^{+}$.
\end{abstract}

\maketitle

\section{Introduction}\setcounter{equation}{0}

In this paper, we study the Cauchy problem of the three-dimensional rotating incompressible Navier--Stokes equations:
\begin{equation}\label{rns}
\begin{cases}
\partial_t u_{\Omega} - \Delta u_{\Omega}  + \Omega e_3 \times u_{\Omega}  + u_{\Omega}  \cdot \nabla u_{\Omega}  + \nabla  \pi_{\Omega}  = 0, \\
\nabla \cdot u_{\Omega}  = 0, \\
u_{\Omega} (0,x) = u_{\Omega,0} ,
\end{cases}
\end{equation}
where \(\Omega \in \mathbb{R}^+\) denotes the rotation rate, \(e_3 = (0,0,1)\) is the unit vector along the vertical direction, and the term \(\Omega e_3 \times u_{\Omega} \) accounts for the Coriolis force. The unknowns are the velocity field \(u_{\Omega}  = u_{\Omega} (t,x) = (u_{\Omega1}, u_{\Omega2} ,u_{\Omega3} )\) and the pressure \( \pi =  \pi(t,x)\), with \(u_{\Omega,0} = (u_{\Omega,01}, u_{\Omega,02} , u_{\Omega,03} )\) representing the initial velocity field.

The rotating Navier-Stokes equations model viscous fluid flow within a reference frame that spins at a constant angular velocity. These equations are foundational for geophysical fluid dynamics (e.g., meteorology and oceanography) and engineering applications involving spinning machinery (see \cite{Pedlosky1987, Majda2003}).

\subsection{Literature reviews}
If \(\Omega = 0\),  the system \eqref{rns} reduces to the classical incompressible Navier-Stokes equatioin, which reads
\begin{equation}\label{INS}
\begin{cases}
\partial_t v - \Delta v + v \cdot \nabla v + \nabla \pi = 0, \\
\nabla \cdot v = 0, \\
v(0,x) = v_0(x).
\end{cases}
\end{equation}
The well-posedness theory for the standard three-dimensional Navier--Stokes equations  has been thoroughly explored. We only recall the results in the critical spaces. The equation \eqref{INS} is invariant under the natural scaling transformation: if \((u, \pi)\) is a solution, then for any \(\lambda > 0\), the rescaled functions
\begin{equation}
(u_\lambda(t,x),\, \pi_\lambda(t,x)) = (\lambda u(\lambda^2 t, \lambda x),\, \lambda^2 \pi(\lambda^2 t, \lambda x))
\end{equation}
also satisfy the system with initial data $\lambda u_0(\lambda x)$. The space $X$ is critical in the sense of this invariance means $\norm{u_0}_X\sim \norm{\lambda u_0(\lambda x)}_X$. Here are some canonical critical spaces:
\EQN{
\dot H^{1/2}\subset L^3 \subset \dot B^{3/p-1}_{p,q} (3\leq p<\infty) \subset BMO^{-1}=\dot F^{-1}_{\infty,2} \subset \dot B^{-1}_{\infty,\infty}.
}
Fujita and Kato (\cite{FujitaKato1964, Kato1984}) established global well-posedness for small initial data in \(\dot{H}^{1/2}\) and \(L^3\), respectively. Cannone (\cite{Cannone1997}), Planchon (\cite{Planchon}) and Chemin (\cite{Chemin}) proved the
well-posedness in $\dot B^{3/p -1}_{p,q}$ for
$1 \leq p < \infty, 1 \leq q \leq\infty$. 
Koch and Tataru \cite{KoTa} proved global well-posedness for small data in $BMO^{-1}$. Well-posedness in some other critical spaces were also obtained, e.g. Fourier-Lebesgue space $\mathcal{X}^{-1}$ by Lei-Lin (\cite{Lei-Lin}), and modulation space $M^{-1}_{\infty,1}$ by Iwabuchi (\cite{Iwabuchi}).
Ill-posedness in $\dot B^{-1}_{\infty,q}$ in the sense that the solution map is discontinuous was shown by 
Bourgain and Pavlovic (\cite{BoPa}) for $q=\infty$, and by Yoneda (\cite{Yoneda}) for $q>2$. Taking notice of the embedding $\dot B^{-1}_{\infty,q}\subset BMO^{-1}$ for $1\leq q\leq 2$, however, this solution is ill-posedness in $\dot B^{-1}_{\infty,q}$ for $1\leq q\leq 2$ by Wang (\cite{Wang}) .
Let us emphasize that the critical Besov spaces with negative index allow us to construct solutions with
highly oscillating initial data such like
$$u_0(x) = \sin{\frac{x_3}{\varepsilon}}
(-\partial_2\Psi(x), \partial_1\Psi(x), 0),$$
where $\Psi\in\mathcal{S}(\R^3)$ and $\varepsilon>0$.

It is natural to extend the results to the rotating case. When rotation is present (\(\Omega \neq 0\)), there are also a lot of studies. We refer to the works (\cite{2,3,4,5,6, 15, 20, 9, 16, 21}).
Hieber and Shibata (\cite{HieberShibata2010}) derived global $L^p-L^q$ estimates for $q <2<p$ and $L^q-H^s$ estimates for $q > 3$ for the Stokes--Coriolis semigroup. These estimates enabled them to prove global well-posedness for small data in $H^{\frac{1}{2}}$.
In \cite{CMZ2013}, Chen, Miao, and Zhang established global well-posedness for small initial data within the hybrid Besov space:
\[
\|u_{0}\|_{\dot{B}_{2,\infty}^{\frac{1}{2}}}^{\ell_{\Omega}} + \|u_0\|_{\dot{B}_{p,\infty}^{-1+\frac{3}{p}}} \ll 1,\quad 2 \leq p \leq 4.
\]
Such space naturally accommodates physically relevant data with highly oscillatory initial data. However, compared to the case $\Omega=0$, the range of $p$ for well-posedness results is considerably restricted.

To the best of our knowledge, the problem of global well-posedness in optimal Besov spaces remains open. This is due to the much intricate analysis for the the Stokes--Coriolis system. Mathematically, it is observed that the Coriolis force induces a dispersive effect, which is captured by the oscillatory Fourier multiplier
\[
e^{\pm i \Omega t \frac{D_3}{|D|}} f(x) = \frac{1}{(2\pi)^3} \int_{\mathbb{R}^3} e^{i x \cdot \xi \pm i \Omega t \frac{\xi_3}{|\xi|}} \widehat{f}(\xi) \, d\xi, \quad (t, x) \in \mathbb{R}^+ \times \mathbb{R}^3.
\]
The Stokes--Coriolis semigroup that governs the linear equation behaves like
$e^{t\Delta\pm i\Omega t \frac{D_3}{|D|}}$ which is parabolic-dispersive.
The dispersive estimates, along with their Strichartz-type variants, play a pivotal role in the analysis of the rotating Navier--Stokes equations. A significant challenge arises from the fact that the Stokes--Coriolis semigroup, fails to be uniformly (in time) bounded in \(L^p\) for \(p \neq 2\) (see Theorem 5 and Theorem 6 in \cite{DPV2006}).

The main objective of this work is to study \eqref{rns} in the hybrid Besov space. We prove a new global-in-time $L^q-L^p$ type dissipative-dispersive estimate for the Stokes--Coriolis system. This allows us to extend the well-posedness theory to a broader class of initial data, by exploiting the framework of hybrid Besov spaces in \cite{CMZ2013}. In addition, we establish the vanishing rotational limit (strong convergence) for the three-dimensional rotating Navier–Stokes equations as $\Omega\rightarrow 0^{+}$.

\subsection{Main results}
Before we present the main result, let us introduce the hybrid Besov space. Let $\chi(\xi)$ be a smooth function valued in $[0,1]$ such that $\chi$ is supported in the ball
$\mathbf{B}(0,\frac{4}{3})=\{\xi\in\mathbb{R}^{3}:|\xi|\leq\frac{4}{3}\}$ and $\varphi(\xi)=\chi(\xi/2)-\chi(\xi)$. We define the homogeneous dyadic blocks and homogeneous low-frequency cut-
off operators:
\begin{eqnarray*}
&&\dot{\Delta}_{j}f:=\varphi(2^{-j}D)f=\mathcal{F}^{-1}(\varphi(2^{-j}\xi)\mathcal{F}f), \quad j\in\mathbb{Z}.
\end{eqnarray*}
Then for any $f\in \mathcal{S}'(\R^3)$, we define
$$f^{\ell_{\alpha}}:=\sum_{2^j< \alpha^{\frac{1}{2}}}\dot{\Delta}_{j} f, \quad f^{h_\alpha}:=\sum_{2^j\geq \alpha^{\frac{1}{2}}}\dot{\Delta}_{j} f.$$
Accordingly, we define the hybrid frequency-restricted Besov semi-norms
\begin{equation*}
\|f\|_{\dot{B}_{p,r}^s}^{\ell_\alpha}:=(\sum_{2^j<\alpha^{\frac{1}{2}}}2^{js}\|\dot{\Delta}_jf\|_{L^p})^{\frac{1}{r}};\quad\|f\|_{\dot{B}_{p,r}^s}^{h_\alpha}:=(\sum_{2^j\geq \alpha^{\frac{1}{2}}}2^{js}\|\dot{\Delta}_jf\|_{L^p})^{\frac{1}{r}}.
\end{equation*}

Our first main result of this paper concerns well-posedness in general $L^q-L^p$ space, which is stated as follows:
\begin{theorem}[Well-posedness]\label{wp}
 Let $2\leq q\leq p\leq 2q\leq\infty$, and assume
     \begin{equation*}
        \frac{9}{4q} - \frac{3}{4p} < 1 < \frac{2}{q} + \frac{1}{p}.
    \end{equation*}
Define
\begin{equation}
\begin{aligned}
&\|u_{0,\Omega}\|_{X_{0,q,p}^{\Omega}}:= \Omega^{3- \frac{6}{q}}\|u_{0,\Omega}\|_{\dot{B}_{q,\infty}^{-7+\frac{15}{q}}}^{\ell_\Omega}
+\|u_{0,\Omega}\|_{\dot{B}_{p,\infty}^{-1+\frac{3}{p}}}^{h_\Omega}<\infty
\end{aligned}
\end{equation}
and
\begin{equation}
\begin{aligned}
&\|u\|_{X_{q,p}^{\Omega}}:=\Omega^{\frac{3}{2}-\frac{3}{q}}
\|u\|^{\ell_{\Omega}}_{\tilde{L}_{T}^{\infty}(\dot{B}_{q,\infty}^{-4+\frac{9}{q}})
\cap\tilde{L}_{T}^{1}(\dot{B}_{q,\infty}^{-2+\frac{9}{q}})}+
\|u\|^{h_{\Omega}}_{\tilde{L}_{T}^{\infty}(\dot{B}_{p,\infty}^{-1+\frac{3}{p}})\cap
\tilde{L}_{T}^{1}(\dot{B}_{p,\infty}^{1+\frac{3}{p}})}.
\end{aligned}
\end{equation}

If the initial data $u_{0,\Omega}\in X_{0,q,p}^{\Omega}$, then there exists a \( T  > 0 \) such that \eqref{rns} has a unique solution
 $u_{\Omega}\in X_{q, p}^{\Omega}$.
Moreover, if \( \|u_{0,\Omega}\|_{X_{0,q,p}^{\Omega}} \) is sufficiently small, then the Cauchy problem \eqref{rns} admits a unique global solution $u_{\Omega}\in X_{q, p}^{\Omega}$ for any \( T>0\).
\end{theorem}

\begin{remark}\label{rk:highosci}

 Let \( 0 \leq \theta < 1 \) and \( q = 2 + \theta \). We arrive at the following valid intervals for the exponent \( p \):
\begin{equation}\label{cns1}
\begin{cases}
p \in [2 + \theta, \, 4 + 2\theta], & 0 \leq \theta < \dfrac{1}{2}, \\[4pt]
p \in \left[2 + \theta, \, 1 + \dfrac{2}{\theta} \right), & \dfrac{1}{2}\leq\theta < 1.
\end{cases}
\end{equation}
In particular, when \( \theta = 0 \), our framework reduces to the classical setting treated in \cite{CMZ2013}. The choice \( \theta= \frac{1}{2} \) gives rise to the broader validity range \( p< 5 \). These developments represent a genuine improvement over earlier works \cite{CMZ2013} by substantially relaxing the constraints on the high-frequency exponent $p$.
\end{remark}

\begin{theorem}[Vanishing rotational limit]\label{strong convergence}
Under the assumptions of Theorem \ref{wp}, let $u_{\Omega}$ be the local solution of \eqref{rns} with initial data $u_{\Omega,0}$, and let $v$ be the
 solution of \eqref{INS} with initial data $v_0$.  
Assume that
\[
\|u_{\Omega,0}-v_0\|_{\dot B_{p,\infty}^{-1+\frac{3}{p}}} \to 0 \quad (\Omega \to 0^+),
\]
then for some $T>0$,
\[
\lim_{\Omega \to 0^+} \bigl\| u_{\Omega} - v \bigr\|_{\tilde{L}_{T}^{\infty}(\dot{B}_{p,\infty}^{-1+\frac{3}{p}})\cap
\tilde{L}_{T}^{1}(\dot{B}_{p,\infty}^{1+\frac{3}{p}})} = 0.
\]
\end{theorem}

\subsection{Sketch of the proof}
We outline the key steps involved in main theorems. The main ingredient of our new $L^q-L^p$ framework is to understand the behavior of the following linearized Stokes-Corioli system:
\begin{equation}
\begin{cases}
\partial_t u_{\Omega} - \Delta u_{\Omega}  + \Omega e_3 \times u_{\Omega}  + \nabla  \pi_{\Omega}  = 0, \\
\nabla \cdot u_{\Omega}  = 0, \\
u_{\Omega} (0,x) = f(x) ,
\end{cases}
\end{equation}
whose associated semi-group is given by $e^{tA_\Omega}$.
The main constraint of the general framework arises because of the dispersive part of the Stokes-Coriolis semi group. This enormously influences the dissipation in the Stokes operator. Indeed, 
in the ``high-frequency" region $|\xi|\geq\Omega^{\frac 12}$, the semigroup $e^{tA_\Omega}$ behaves predominantly as the heat semigroup $e^{ct\Delta}$. However, in the ``low-frequency" region $|\xi|<\Omega^{\frac 12}$, dispersive  operator $e^{\pm i\Omega t \frac{D_{3}}{|D|}}$ is prominent, which prevented earlier approaches in \cite{ HieberShibata2010,5,6, 15,20} to work in general $L^p$ spaces. In \cite{CMZ2013}, $L^2$-based estimates for low frequency was used to deal with the dispersive part, but the consequence is the restriction of the range for $p\leq 4$ from the high-high to low estimates.

In this paper, we follow the ideas in our previous work regarding the compressible Navier-Stokes equation \cite{GSY}. We use $L^q$-estimates for the low frequency, but with a different dispersion relation. Using the fixed time estimate obtained in \cite{BV2023}, we derive a new estimate for Stokes-Coriolis semigroup that incorporates dispersive effects within different frequency regions:
\begin{equation}\label{1.23}
\|\dot{\Delta}_{j} e^{tA_\Omega} f\|_{L^p}\lesssim\left\{
\begin{aligned}
&e^{-c t 2^{2j}}\|\dot{\Delta}_{j}f\|_{L^p},\quad\quad\quad\quad \quad\quad\quad\quad\quad2^j\geq\Omega^{\frac 12},\\ 
&e^{-c t 2^{2j}}\Omega^{\,3|\frac12 - \frac1p|}
        2^{-6k|\frac12 - \frac1p|}\|\dot{\Delta}_{j} f\|_{L^p},\quad\,\,\, 2^j<\Omega^{\frac 12}.
\end{aligned}
\right.
\end{equation}
Above inequalities indicate us to introduce the hybrid Besov space, with frequencies cut-off depending on parameter $\Omega$, and extend the Lebesgue framework to the more general one to the previously unattainable range \( 2<q<\frac p2\) by imposing the supercritical regularity in the low frequencies.
The remaining task is to establish the bilinear estimates in the new $L^q-L^p$ hybrid space. The arguments are similar as the estimates in our previous work \cite{GSY}. However, we tracked carefully the dependence on the parameter $\Omega$, see Proposition \ref{besti}.

The rest of this paper unfolds as follows. In Section 2, we
provide a detailed analysis to derive new $L^q-L^p$ bounds for the Stokes-Coriolis system,
which is crucial to this paper. In Section 3, we prove the well-posedness part of Theorem \ref{wp}. Section 4 is devoted to prove the limiting behaviour. In Appendix, we recall the classical Littlewood-Paley theory.

\section{Estimates of the Stokes-Coriolis system}
In this section, we establish the  estimate in terms of the linearized system. We consider the linear Stokes-Coriolis system
\begin{equation}\label{linear}
\begin{cases}
\partial_t u_{\Omega} - \Delta u_{\Omega}  + \Omega e_3 \times u_{\Omega}  + \nabla  \pi_{\Omega}  = g, \\
\nabla \cdot u_{\Omega}  = 0, \\
u_{\Omega} (0,x) = f(x) ,
\end{cases}
\end{equation}
and $e^{tA_\Omega}$ to be the corresponding Stokes-Coriolis semigroup. The main goal of this subsection is to prove the following proposition:
\begin{proposition} \label{linearesti}
  Let $s \in \mathbb{R}$ and $p,r,\rho,\rho_1 \in [1, \infty]$ and $\rho\leq\rho_1$. Suppose $f,g$ are distributions and we define
  $$U\triangleq e^{tA_\Omega}f;\quad V\triangleq\int_0^t e^{(t-\tau)A_\Omega} g(\tau)  d\tau,$$
then

(1). There exists a constant $C > 0$, independent of  $\Omega$, such that for all $t \geq 0$,
\begin{equation}\label{high-estimate1}
\|U\|_{\tilde{L}_{t}^{\rho_{1}}(\dot{B}_{p,r}^{s + \frac{2}{\rho_{1}}})}^{h_{\Omega}}
\leq C\|f\|_{\dot{B}_{p,r}^{s}}^{h_{\Omega}};
\end{equation}
\begin{equation}\label{high-estimate2}
\|V\|_{\tilde{L}_{t}^{\rho_{1}}(\dot{B}_{p,r}^{s + \frac{2}{\rho_{1}}})}^{h_{\Omega}}
\leq C\|g\|_{\tilde{L}_{t}^{\rho}(\dot{B}_{p,r}^{s - 2 + \frac{2}{\rho}})}^{h_{\Omega}}.
\end{equation}

(2). There exists a constant $C > 0$, independent of $\Omega$, such that for all $t \geq 0$,
\begin{equation}\label{low-estimate1}
\begin{aligned}
\|U\|_{\tilde{L}_t^{\rho_1}(\dot{B}_{p,r}^{s+ \frac{2}{\rho_1}})}^{\ell_{\Omega}}
\leq C\bigl(\|f\|_{\dot{B}_{p,r}^{s}}^{\ell_{\Omega}} +
\Omega^{3|\frac{1}{2} - \frac{1}{p})|}\|f\|_{\dot{B}_{p,r}^{s-3+\frac{6}{p}}}^{\ell_{\Omega}}\bigr);
\end{aligned}
\end{equation}
\begin{equation}\label{low-estimate2}
\begin{aligned}
\|V\|_{\tilde{L}_t^{\rho_1}(\dot{B}_{p,r}^{s+ \frac{2}{\rho_1}})}^{\ell_{\Omega}}
\leq C\bigl(
\|g\|_{\tilde{L}_t^\rho(\dot{B}_{,r}^{s - 2 + \frac{2}{\rho}})}^{\ell_{\Omega}}+\Omega^{3|\frac{1}{2} - \frac{1}{p}|}\|g\|_{\tilde{L}_t^\rho(\dot{B}_{p,r}^{s-5+\frac{6}{p}+ \frac{2}{\rho}})}^{\ell_{\Omega}}\bigr).
\end{aligned}
\end{equation}

(3). Assume that $\rho_1 \in [1, \infty)$. If $\|f\|_{\dot{B}_{p,r}^{s}}^{h_{\Omega}}<\infty$, then \begin{equation}\label{Tlim}
\begin{aligned}
\lim_{T \to  0^{+}} \| U \|_{\tilde L_T^{\rho_1}(\dot B_{p,r}^{s+\frac{2}{\rho_1}})}^{h_{\Omega}} = 0;
\end{aligned}
\end{equation}
and if $$\|f\|_{\dot{B}_{p,r}^{s}}^{\ell_{\Omega}} +
\Omega^{3|\frac{1}{2} - \frac{1}{p}|}\|f\|_{\dot{B}_{p,r}^{s-3+\frac{6}{p}}}^{\ell_{\Omega}}<\infty,$$ then
\begin{equation}
\begin{aligned}
\lim_{T \to  0^{+}} \| U \|_{\tilde L_T^{\rho_1}(\dot B_{p,r}^{s+\frac{2}{\rho_1}})}^{^{\ell_{\Omega}}} = 0.
\end{aligned}
\end{equation}
\end{proposition}

The rest of this section would be devoted to proving above Proposition. We shall give a careful analysis on the corresponding Green matrix and give the estimates by delicate frequency decomposition.

\subsection{Analysis of Green matrix}
Let \( P(\xi) \) denote the Fourier multiplier matrix associated with the Helmholtz projection \(\mathbb{P}\), defined by
\[
\widehat{\mathbb{P} f}(\xi) = P(\xi) \hat{f}(\xi), \quad \text{where} \quad P(\xi) := \left( \delta_{jk} - \frac{\xi_j \xi_k}{|\xi|^2} \right)_{1 \leq j,k \leq 3} 
\]
for \(\xi \in \mathbb{R}^3 \setminus \{0\}\). By \eqref{linear}, the Stokes--Coriolis semigroup admits the explicit representation
\[
\begin{aligned}
e^{tA_\Omega}f &= \mathcal{F}^{-1} \left[ e^{-t|\xi|^2} \left\{ \cos \left( \Omega \frac{\xi_3}{|\xi|} t \right) I + \sin \left( \Omega \frac{\xi_3}{|\xi|} t \right) R(\xi) \right\} \hat{f}(\xi) \right]
\end{aligned}
\]
for divergence-free vector fields \( f \in L^2 (\mathbb{R}^3)^3 \). Here, \( I \) denotes the \( 3 \times 3 \) identity matrix, and \( R(\xi) \) is the skew-symmetric matrix associated with Riesz transforms, given by
\[
R(\xi) := \frac{1}{|\xi|} \begin{pmatrix} 0 & \xi_3 & -\xi_2 \\ -\xi_3 & 0 & \xi_1 \\ \xi_2 & -\xi_1 & 0 \end{pmatrix}, \quad \xi \in \mathbb{R}^3 \setminus \{0\}.
\]

\begin{lemma}[\cite{BV2023}]\label{p5}
For $\Phi \in C^{\infty}(\mathbb{R}^3 \setminus \{0\}, \mathbb{R})$,
 a homogeneous function of degree 0, we consider a family of Fourier multiplier operators indexed by $t \in \mathbb{R}$, defined on the set of Schwartz functions $\mathcal{S}(\mathbb{R}^3)$ with
\[
T_{\Phi}^t f(x) := \int_{\mathbb{R}^3}
e^{ i (x \cdot \xi + t \Phi(\xi))} \widehat{f}(\xi) \, d\xi. 
\]
 There exists a finite constant $C_{d,\Phi}$ such that for every exponent $p \in (1, \infty)$ and $t \in \mathbb{R}$ we have
\[
\left\| T_{\Phi}^{t} \right\|_{L^{p} \to L^{p}} \leq C_{\Phi} (p^{*} - 1) \langle t \rangle^{3| \frac{1}{2} - \frac{1}{p} |}.
\]
where $p^{*} := \max\left\{p, \frac{p}{p-1}\right\}$ and $\langle t\rangle=\sqrt{1+t^{2}}$.
\end{lemma}

\begin{lemma}\label{stokeslow}
    Assume $1 < p < \infty$. Then
    \begin{equation}\label{lowfresemi}
        \| e^{tA_\Omega} \dot{\Delta}_{k} f \|_{L^{p}}
        \lesssim  (p^{*} - 1) \, e^{-c t 2^{2k}}
        \Bigl(1 + \Omega^{\,3|\frac12 - \frac1p|}
        2^{-6k|\frac12 - \frac1p|} \Bigr)
        \| \dot{\Delta}_k f \|_{L^p}.
    \end{equation}
\end{lemma}

\begin{proof}
The operator $e^{tA_\Omega}\dot{\Delta}_k$ is a Fourier multiplier with symbol
    \[
        M_{t,k}(\xi) = e^{-t|\xi|^2} e^{it\Omega \cdot R(\xi)} \psi_k(\xi),
    \]
where $\psi_k$ corresponds to the frequency localization $\dot{\Delta}_k$, which restricts to $|\xi| \sim 2^k$. By Lemma~\ref{p5},
    \[
        \| e^{tA_\Omega} \dot{\Delta}_{k} f \|_{L^{p}}
        \lesssim (p^{*} - 1) \, e^{-c t 2^{2k}}
        \langle \Omega t \rangle^{3|\frac12 - \frac1p|}
        \| \dot{\Delta}_{k} f \|_{L^{p}}.
    \]
Using the bound
    \[
        \langle \Omega t \rangle \lesssim 1 + |\Omega| 2^{-2k} \cdot 2^{2k} t
    \]
and noting that the exponential decay dominates the polynomial growth, we obtain the simplified estimate
    \[
        \| e^{tA_\Omega} \dot{\Delta}_k f \|_{L^p}
        \lesssim (p^{*} - 1) \, e^{-c t 2^{2k}}
        \Bigl(1 + \Omega^{\,3|\frac12 - \frac1p|}
        2^{-6k|\frac12 - \frac1p|} \Bigr)
        \| \dot{\Delta}_k f \|_{L^p}.
    \]
This completes the proof.
\end{proof}

\subsection{Proof of Proposition  ~\ref{linearesti}}
In the following, we begin to prove  Proposition  ~\ref{linearesti}.
By Lemma~\ref{stokeslow}, we have
\begin{equation}\label{semi1}
\begin{aligned}
\|\dot{\Delta}_k U\|_{L^q}
&\lesssim e^{-ct 2^{2k}} \bigl(\Omega^{3|\frac{1}{2} - \frac{1}{q}|}2^{-6k|\frac{1}{2}-\frac{1}{q}|} + 1\bigr) \|\dot{\Delta}_k f\|_{L^q} \
\end{aligned}
\end{equation}
and
\begin{equation}\label{semi2}
\begin{aligned}
\|\dot{\Delta}_k V\|_{L^q}
&\lesssim \int_0^t e^{-c(t-s) 2^{2k}} \bigl(1 + \Omega^{3|\frac{1}{2} - \frac{1}{q}|}2^{-6k|\frac{1}{2}-\frac{1}{q}|}\bigr) \|\dot{\Delta}_k g(s)\|_{L^q}  ds.
\end{aligned}
\end{equation}
Taking the $L^{\rho_1}$-norm in time in \eqref{semi1} and \eqref{semi2} gives
\begin{equation}\label{semi3}
\begin{aligned}
\|\dot{\Delta}_k U\|_{L_t^{\rho_1}(L^q)} \lesssim \biggl(\frac{1 - e^{-c t \rho_1 2^{2k}}}{c \rho_1 2^{2k}} \biggr)^{\!\!\frac{1}{\rho_1}}
\bigl( \|\dot{\Delta}_kf\|_{L^q} + \Omega^{3|\frac{1}{2}-\frac{1}{q})|}2^{-6k|\frac{1}{2}-\frac{1}{q}|}\|\dot{\Delta}_k f\|_{L^q} \bigr) \\
\end{aligned}
\end{equation}
and
\begin{equation}\label{semi4}
\begin{aligned}
\|\dot{\Delta}_k V\|_{L_t^{\rho_1}(L^q)} \lesssim  \biggl( \frac{1 - e^{-c t \rho_2 2^{2k}}}{c \rho_2 2^{2k}} \biggr)^{\!\!\frac{1}{\rho_2}}
\bigl( \|\dot{\Delta}_k g\|_{L_t^\rho(L^q)} + \Omega^{3|\frac{1}{2} - \frac{1}{q}|}2^{-6k|\frac{1}{2}-\frac{1}{q}|}\|\dot{\Delta}_k g\|_{L_t^\rho(L^q)} \bigr),
\end{aligned}
\end{equation}
where $\frac{1}{\rho_2} = 1 + \frac{1}{\rho_1} - \frac{1}{\rho}$. Let $k_0 = \max\{k : 2^{2k}<\Omega\}$.

\underline{(1) High-frequency estimate.}  When $k>k_{0}$, we have $\Omega^{3|\frac{1}{2} - \frac{1}{q}|}2^{-6k|\frac{1}{2}-\frac{1}{q}|}\lesssim 1$. Multiplying  \eqref{semi3} and
\eqref{semi4} by $2^{k(s + 2/\rho_1)}$ and taking the $\ell^r$-norm for $k>k_{0}$, we get \eqref{high-estimate1} and
\eqref{high-estimate2}.

\underline{(2) Low-frequency estimate.}  
  Multiplying  \eqref{semi3} and
\eqref{semi4} by $2^{k(s + 2/\rho_1)}$ and taking the $\ell^r$-norm for $k\leq k_{0}$, we obtain \eqref{low-estimate1} and \eqref{low-estimate2}.

\underline{(3) Limit behavior as $T \to  0^{+}$.} 
When $\rho_1 \in [1, \infty)$. We split the frequency axis into low and high frequencies, and choose appropriate frequency thresholds and a small time interval. Applying 
\eqref{semi1}, there exists $J_2$ (independent of $T$) large enough such that
\begin{equation}\label{high0}
\begin{aligned}
\sum_{j\ge J_2} \Bigl(2^{j(s+\frac{2}{\rho_1})}\|\dot\Delta_j e^{t A_\Omega} f\|_{L_T^{\rho_1}(L^p)}\Bigr)^r
\lesssim\sum_{j\ge J_2} 2^{j s r}\|\dot\Delta_j f\|_{L^p}^r
< \frac{\varepsilon}{2}.
\end{aligned}
\end{equation}
For fixed $J_2$. Applying 
\eqref{semi1}, then
\begin{equation}\label{high1}
\begin{aligned}
&\sum_{\sqrt{\Omega}\leqslant 2^{j}<2^{J_2}} \Bigl(2^{j(s+\frac{2}{\rho_1})}\|\dot\Delta_j e^{t A_\Omega} f\|_{L_T^{\rho_1}(L^p)}\Bigr)^r
\lesssim\bigl(1 - e^{-c T \rho_1 2^{2J_2}})^r \sum_{\sqrt{\Omega}\leqslant 2^{j}<2^{J_2}} 2^{j s r}\|\dot\Delta_j f\|_{L^p}^r
\end{aligned}
\end{equation}
and
\begin{equation}\label{lowmid0}
\begin{aligned}
&\sum_{2^{2j}<\Omega} \Bigl(2^{j(s+\frac{2}{\rho_1})}\|\dot\Delta_j e^{t A_\Omega} f\|_{L_T^{\rho_1}(L^p)}\Bigr)^r\\
&\lesssim (1 - e^{-c T \rho_1 2^{2J_2}})^{r}\,\sum_{2^{2j}<\Omega} 2^{jrs}\|\dot\Delta_j f\|_{L^p}^{r}
\cdot \bigl(1+\Omega^{3r|\frac12-\frac1p|}2^{-6jr|\frac12-\frac1p|}\bigr),
\end{aligned}
\end{equation}
from which, choose $T$ small enough and $J_2$ (independent of $T$) large enough, with aid of \eqref{high0} and \eqref{high1}, we have
\begin{equation}\label{high}
\begin{aligned}
&\Bigl(\sum_{\Omega\leqslant 2^{2j}} 2^{jr(s+\frac{2}{\rho_1})}\|\dot\Delta_j e^{t A_\Omega} f\|_{L_T^{\rho_1}(L^p)}^{r}\Bigr)^{\frac{1}{r}}
\leq C\bigl(1 - e^{-c T \rho_1 2^{2J_2}})\|f\|_{\dot{B}_{p,r}^{s}}^{h_{\Omega}}+\frac{\varepsilon}{2}< \varepsilon
\end{aligned}
\end{equation}
and by \eqref{lowmid0} we have
\begin{equation}\label{lowmid}
\begin{aligned}
&\Bigl(\sum_{2^{2j}<\Omega} 2^{jr(s+\frac{2}{\rho_1})}\|\dot\Delta_j e^{t A_\Omega} f\|_{L_T^{\rho_1}(L^p)}^{r}\Bigr)^{\frac{1}{r}}\\
&\lesssim \bigl(1 - e^{-c T \rho_1 2^{2J_2}})\bigl(\|f\|_{\dot{B}_{p,r}^{s}}^{\ell_{\Omega}} +
\Omega^{3|\frac{1}{2} - \frac{1}{p}|}\|f\|_{\dot{B}_{p,r}^{s-3+\frac{6}{p}}}^{\ell_{\Omega}}\bigr)< \varepsilon.
\end{aligned}
\end{equation}
Summing up \eqref{high} and \eqref{lowmid}, we obtained the desired results.
\section{The Proof of Theorem \ref{wp}}

In this section, we aim at proving Theorem~\ref{wp}. For simplify the expression,  we use $\ell,\, h$ to denote $\ell_\Omega,\, h_{\Omega}$. Let us begin with the well-posedness. 

By Duhamel formula, the solution of rotating Navier-Stokes is given by \[
u_{\Omega}(t) = e^{tA_\Omega} u_{0,\Omega} +\Psi(u_{\Omega},u_{\Omega}),\]
where
$$\Psi(u_{\Omega},u_{\Omega})=- \int_0^t e^{(t-\tau)A_\Omega} \PP \nabla \cdot (u_{\Omega} \otimes u_{\Omega})(\tau) \, d\tau.$$ Consequently in terms of the classical Kato's solution, we state the following uniform bilinear estimates, which is the cornerstone of the well-posedness:
\begin{proposition}\label{besti}
    Let $f,\, g$ to be distributions. Let $2 \leq q \leq p \leq 2q\leq\infty$, and assume
    \begin{equation*}
        \frac{9}{4q} - \frac{3}{4p} < 1 < \frac{2}{q} + \frac{1}{p}.
    \end{equation*}
 Then
    \begin{equation*}
        \begin{aligned}
                  \bigl\|\Psi(f,g)\bigr\|_{X_{q,p}^{\Omega}} 
             \lesssim \|f\|_{X_{q,p}^{\Omega}}\|g\|_{X_{q,p}^{\Omega}}.
        \end{aligned}
    \end{equation*}
\end{proposition}

The rest of this section would be mostly devoted to the proof of above Proposition. In light of uniform bilinear estimates, we may obtain the well-posedness by a standard fixed point argument.
\subsection{Proof of Proposition \ref{besti}}

In light of Proposition  ~\ref{linearesti} with taking $p=q$, $s=\frac{3}{q}-1$ in high frequencies and $s=\frac{9}{q}-4$ in low frequencies, there holds
\begin{equation*}
\begin{aligned}
\bigl\|\Psi(f,g)\bigr\|_{X_{q,p}^{\Omega}}\lesssim \|\PP \nabla \cdot (f\otimes g)\|_{\tilde{L}_t^1(\dot{B}_{q,\infty}^{\frac{3}{q}-1})}+\Omega^{3 - \frac{6}{q})}\|\PP \nabla \cdot (f\otimes g)\|^{\ell}_{\tilde{L}_t^1(\dot{B}_{q,\infty}^{\frac{15}{q}-7})}.
\end{aligned}
\end{equation*}
Apparently Proposition \ref{besti} is the direct result of the following two lemmas.
\begin{lemma}\label{uulow}
Let $f$ and $g$ be distributions, $2\le q\le p\le 2q$ and $\frac{3}{p}-\frac{3}{q}+1>0$. Then
\begin{equation*}
    \|\PP \nabla \cdot (f \otimes g)\|_{\tilde{L}_{t}^{1}(\dot{B}_{q,\infty}^{\frac{3}{q}-1})}\lesssim\|f\|_{X_{q,p}^{\Omega}} \|g\|_{X_{q,p}^{\Omega}}.
\end{equation*}
\end{lemma}

\begin{lemma}\label{uuulow}
 Let $f,g$ to be distributions and $2 \leq q \leq p \leq 2q$, assume $\frac{9}{4q} - \frac{3}{4p}< 1 < \frac{2}{q} + \frac{1}{p}$. Then
\begin{equation*}
\begin{aligned}
\Omega^{\,3-\frac{6}{q}}
\bigl\|\PP \nabla \cdot (f \otimes g)\bigr\|^{\ell}_{\tilde{L}_{t}^{1}
(\dot{B}_{q,\infty}^{\frac{15}{q}-7})}\lesssim\|f\|_{X_{q,p}^{\Omega}} \|g\|_{X_{q,p}^{\Omega}}.
        \end{aligned}
    \end{equation*}
\end{lemma}

\subsubsection{Proof of Lemma \ref{uulow}}
 First, the Bony decomposition implicates that
\begin{equation}\label{mm1}
    \|\PP \nabla \cdot (f \otimes g)\|_{\tilde{L}_{t}^{1}(\dot{B}_{q,\infty}^{-1+\frac{3}{q}})}
    \lesssim \|T_f g \|_{\tilde{L}_{t}^{1}(\dot{B}_{q,\infty}^{\frac{3}{q}})}+\|T_g f \|_{\tilde{L}_{t}^{1}(\dot{B}_{q,\infty}^{\frac{3}{q}})}+\|R(f,g) \|_{\tilde{L}_{t}^{1}(\dot{B}_{q,\infty}^{\frac{3}{q}})}.
\end{equation}
where
 \[
\dot{\Delta}_{k}T_{f}g        = \dot{\Delta}_{k}\Bigl(\sum_{k'}\dot{S}_{k'-1}f\;\dot{\Delta}_{k'}g\Bigr)
        = \sum_{|k-k'|\leq 4}\dot{\Delta}_{k}\bigl(\dot{S}_{k'-1}f\;\dot{\Delta}_{k'}g\bigr);
    \]
     \[
        \dot{\Delta}_{k}R(f,g)
        = \dot{\Delta}_{k}\Bigl(\sum_{k'}\tilde{\dot{\Delta}}_{k'}f\,
           \dot{\Delta}_{k'}g\Bigr)
        = \sum_{k \leq k'+2} \dot{\Delta}_{k}\bigl(\tilde{\dot{\Delta}}_{k'}f\,
           \dot{\Delta}_{k'}g\bigr).
    \]
For $q \leq p$, set $\frac{1}{q} = \frac{1}{\tilde{p}} + \frac{1}{p}$. The assumption $p \leq 2q$ implies $p \leq \tilde{p}$, and we thus obtain
\begin{align*}
    \| \dot{\Delta}_{k} T_{f} g \|_{L^q}
    &\lesssim \sum_{|k-k'|\leq 4} \| \dot{S}_{k'-1} f \dot{\Delta}_{k'} g \|_{L^q} \\
    &\lesssim \sum_{|k-k'|\leq 4} \sum_{j \leq k'-2} \| \dot{\Delta}_{j}f \|_{L^{\tilde{p}}} \| \dot{\Delta}_{k'}g \|_{L^p} \\
    &\lesssim \sum_{|k-k'|\leq 4} \left( \sum_{j \leq k'-2} 2^{(\frac{3}{p}-\frac{3}{q}+1)j} \right)
        \|f \|_{\dot{B}^{-1+\frac{3}{p}}_{p,\infty}} \| \dot{\Delta}_{k'} g \|_{L^p}.
\end{align*}
Applying Hölder's inequality, we obtain for $\frac{3}{p}-\frac{3}{q}+1>0$,
\begin{equation}\label{P1}
    \begin{aligned}
        \|T_f g \|_{\tilde{L}_{t}^{1}(\dot{B}_{q,\infty}^{\frac{3}{q}})}
        &\lesssim \|f \|_{\tilde{L}_{t}^{\infty}(\dot{B}_{p,\infty}^{\frac{3}{p}-1})}\|g \|_{\tilde{L}_{t}^{1}(\dot{B}_{p,\infty}^{1+\frac{3}{p}})} \lesssim\|f\|_{X_{q,p}^{\Omega}} \|g\|_{X_{q,p}^{\Omega}},
    \end{aligned}
\end{equation}
where we apply in the last inequality that
\begin{equation}
\begin{aligned}
\|f \|_{\tilde{L}_{t}^{\infty}(\dot{B}_{p,\infty}^{\frac{3}{p}-1})}&\lesssim\|f \|^{\ell}_{\tilde{L}_{t}^{\infty}(\dot{B}_{q,\infty}^{\frac{3}{q}-1})}
+\|f \|^{h}_{\tilde{L}_{t}^{\infty}(\dot{B}_{p,\infty}^{\frac{3}{p}-1})}\\&\lesssim\Omega^{\frac{3}{2}-\frac{3}{q}}
\|f\|^{\ell}_{\tilde{L}_{t}^{\infty}(\dot{B}_{q,\infty}^{-4+\frac{9}{q}})}+\|f \|^{h}_{\tilde{L}_{t}^{\infty}(\dot{B}_{p,\infty}^{\frac{3}{p}-1})}\lesssim\|f\|_{X_{q,p}^{\Omega}},
    \end{aligned}
\end{equation}
and similarly
\begin{equation}
\begin{aligned}
\|g \|_{\tilde{L}_{t}^{1}(\dot{B}_{p,\infty}^{\frac{3}{p}+1})}\lesssim\Omega^{\frac{3}{2}-\frac{3}{q}}\|g\|^{\ell}_{\tilde{L}_{t}^{1}(\dot{B}_{q,\infty}^{-2+\frac{9}{q}})}+\|g\|^{h}_{\tilde{L}_{t}^{1}(\dot{B}_{p,\infty}^{1+\frac{3}{p}})}\lesssim\|g\|_{X_{q,p}^{\Omega}}.
    \end{aligned}
\end{equation}
Symmetrically we can handle the other paraproduct term. As for the remainder, we have
\begin{align*}
    \| \dot{\Delta}_{k} R(f,g) \|_{L^q}
    &\lesssim 2^{(\frac{6}{p}-\frac{3}{q})k}
        \left\| \dot{\Delta}_{k} \left( \sum_{k'} \tilde{\dot{\Delta}}_{k'} f \dot{\Delta}_{k'} g \right) \right\|_{L^{p/2}} \\
    &\lesssim 2^{(\frac{6}{p}-\frac{3}{q})k} \sum_{k' \geq k-2}
        2^{(-\frac{6}{p})k'} 2^{(\frac{3}{p}-1)k'} \| \dot{\Delta}_{k'} f \|_{L^{p}}
        \cdot 2^{(1+\frac{3}{p})k'}\|\dot{\Delta}_{k'} g \|_{L^p},
\end{align*}
hence Hölder's inequality yields
\begin{equation}\label{R1}
    \begin{aligned}
        \|R(f, g) \|_{\tilde{L}_{t}^{1}(\dot{B}_{q,\infty}^{\frac{3}{q}})}
        &\lesssim \|f \|_{\tilde{L}_{t}^{\infty}(\dot{B}_{p,\infty}^{\frac{3}{p}-1})}\|g \|_{\tilde{L}_{t}^{1}(\dot{B}_{p,\infty}^{1+\frac{3}{p}})} \lesssim\|f\|_{X_{q,p}^{\Omega}} \|g\|_{X_{q,p}^{\Omega}}.
    \end{aligned}
\end{equation}
Consequently \eqref{P1} and \eqref{R1} lead to Lemma \ref{uulow}.


\subsubsection{Proof of Lemma \ref{uuulow}}
Again using Bony's decomposition, we have
\begin{equation}\label{mm00}
    \|\PP \nabla \cdot (f \otimes g)\|^\ell_{\tilde{L}_{t}^{1}(\dot{B}_{q,\infty}^{\frac{15}{q}-7})}
    \lesssim \|T_f g \|^\ell_{\tilde{L}_{t}^{1}(\dot{B}_{q,\infty}^{\frac{15}{q}-6})}+\|T_g f \|^\ell_{\tilde{L}_{t}^{1}(\dot{B}_{q,\infty}^{\frac{15}{q}-6})}+\|R(f,g) \|^\ell_{\tilde{L}_{t}^{1}(\dot{B}_{q,\infty}^{\frac{15}{q}-6})}.
\end{equation}   
We begin with paraproduct $T_f g^{\ell}$,
For $q \leq p$, set $\frac{1}{q} = \frac{1}{\tilde{p}} + \frac{1}{p}$.  Since $p \leq 2q$, we obtain $p \le \tilde{p}$, and consequently
    \begin{align*}
\|\dot{\Delta}_{k}T_{f}g^{\ell}\|_{L^{q}}
        \lesssim \sum_{|k-k'|\leq 4}
            \Bigl(\sum_{j \leq k'-2} 2^{(\frac{3}{p} - \frac{3}{q} + 4 - \frac{6}{q})j}\Bigr)
            \|f\|_{\dot{B}^{\frac{3}{p}-4+\frac{6}{q}}_{p,\infty}}
            \|\dot{\Delta}_{k'}g^{\ell}\|_{L^{p}} .
    \end{align*}
For $q \ge 2$, $p \ge q$, and $\frac{9}{q} - \frac{3}{p}< 4$, we obtain
    \begin{equation}\label{mm11}
        \begin{aligned}
            \|T_{f}g^{\ell}\|_{\tilde{L}_{t}^{1}(\dot{B}_{q,\infty}^{\frac{15}{q}-6})}^{\ell}
            &\lesssim \|f\|_{\tilde{L}_{t}^{\infty}(\dot{B}_{p,\infty}^{\frac{3}{p}+\frac{6}{q}-4})}
            \|g^{\ell}\|_{\tilde{L}_{t}^{1}(\dot{B}_{p,\infty}^{-2+\frac{3}{p}+\frac{6}{q}})} \\
            &\lesssim \Bigl(\Omega^{-3|\frac{1}{2}-\frac{1}{q}|}
            \|f\|_{\tilde{L}_{t}^{\infty}(\dot{B}_{p,\infty}^{\frac{3}{p}-1})}^{h}
            + \|f\|_{\tilde{L}_{t}^{\infty}(\dot{B}_{q,\infty}^{\frac{9}{q}-4})}^{\ell}\Bigr)
\|g^{\ell}\|_{\tilde{L}_{t}^{1}(\dot{B}_{q,\infty}^{-2+\frac{9}{q}})}.
        \end{aligned}
    \end{equation}
Next for $T_f g^{h}$. When $q \leq p\leq 2q$, 
    \begin{align*}
\|\dot{\Delta}_{k}T_{f}g^{h}\|_{L^{q}}
        \lesssim \sum_{|k-k'|\leq 4}
            \Bigl(\sum_{j \leq k'-2} 2^{(\frac{3}{p} - \frac{3}{q} + 1)j}\Bigr)
            \|f\|_{\dot{B}^{\frac{3}{p}-1}_{p,\infty}}
            \|\dot{\Delta}_{k'}g^{h}\|_{L^{p}},
    \end{align*}
from which, for  $q \ge 2$ and $\frac{3}{q} - \frac{3}{p}<1$, we obtain
    \begin{equation}
        \begin{aligned}
            \|T_{f}g^{h}\|_{\tilde{L}_{t}^{1}(\dot{B}_{q,\infty}^{\frac{15}{q}-6})}^{\ell}
            &\lesssim \|f\|_{\tilde{L}_{t}^{\infty}(\dot{B}_{p,\infty}^{\frac{3}{p}-1})}
    \|g^{h}\|_{\tilde{L}_{t}^{1}(\dot{B}_{p,\infty}^{-5+\frac{3}{p}+\frac{12}{q}})} \\
            &\lesssim \Bigl(
            \|f\|_{\tilde{L}_{t}^{\infty}(\dot{B}_{p,\infty}^{\frac{3}{p}-1})}^{h}
            + \Omega^{\frac{3}{2}-\frac{3}{q}}\|f\|_{\tilde{L}_{t}^{\infty}(\dot{B}_{q,\infty}^{\frac{9}{q}-4})}^{\ell}\Bigr)
            \Omega^{-3+\frac{6}{q}}\|g^{h}\|_{\tilde{L}_{t}^{1}(\dot{B}_{p,\infty}^{1+\frac{3}{p}})}.
        \end{aligned}
    \end{equation}
Similar to estimate $T_f g$, we also obtain
    \begin{equation}\label{mm12}
        \begin{aligned}
            &\|T_{g}f\|_{\tilde{L}_{t}^{1}(\dot{B}_{q,\infty}^{\frac{15}{q}-6})}^{\ell}\lesssim \Bigl(
            \|g\|_{\tilde{L}_{t}^{\infty}(\dot{B}_{p,\infty}^{\frac{3}{p}-1})}^{h}
            + \Omega^{\frac{3}{2}-\frac{3}{q}}\|g\|_{\tilde{L}_{t}^{\infty}(\dot{B}_{q,\infty}^{\frac{9}{q}-4})}^{\ell}\Bigr)
            \Omega^{-3+\frac{6}{q}}\|f^{h}\|_{\tilde{L}_{t}^{1}(\dot{B}_{p,\infty}^{1+\frac{3}{p}})}\\
            & \ \ \ \ \ \ \ \ \quad \quad +\Bigl(\Omega^{-\frac{3}{2}+\frac{3}{q}}
            \|g\|_{\tilde{L}_{t}^{\infty}(\dot{B}_{p,\infty}^{\frac{3}{p}-1})}^{h}
            + \|g\|_{\tilde{L}_{t}^{\infty}(\dot{B}_{q,\infty}^{\frac{9}{q}-4})}^{\ell}\Bigr)
            \|f^{\ell}\|_{\tilde{L}_{t}^{1}(\dot{B}_{q,\infty}^{-2+\frac{9}{q}})}.
        \end{aligned}
    \end{equation}
For the remainder term, when $\frac{2}{q} + \frac{1}{p} > 1$ and $2 \leq q \leq p \leq 2q$. Spectral localization yields
   $\tilde{\dot{\Delta}}_{k'}f^{h} \triangleq \sum_{|k-k'|\leq 1}\dot{\Delta}_{k}f^{h}$. For $2 \leq p \leq 2q$, H\"{o}lder's and Bernstein's inequalities give
    \begin{align*}
        \|\dot{\Delta}_{k}R(f^{h},g^{h})\|_{L^{q}}
        &\lesssim 2^{(\frac{6}{p}-\frac{3}{q})k}
            \Bigl\|\dot{\Delta}_{k}
            \Bigl(\sum_{k'}\tilde{\dot{\Delta}}_{k'}f^{h}\,
            \dot{\Delta}_{k'}g^{h}\Bigr)\Bigr\|_{L^{p/2}} \\
        &\lesssim 2^{(\frac{6}{p}-\frac{3}{q})k}
            \sum_{k'\geq k-2} \|\dot{\Delta}_{k'}f^{h}\|_{L^{p}}
            \|\dot{\Delta}_{k'}g^{h}\|_{L^{p}} \\
        &\lesssim 2^{(\frac{6}{p}-\frac{3}{q})k}
            \sum_{k'\geq k-2} 2^{(-\frac{6}{p}-\frac{12}{q}+6)k'}
            2^{(-1+\frac{3}{p})k'}\|\dot{\Delta}_{k'}f^{h}\|_{L^{p}}\;
            2^{(\frac{3}{p}+\frac{12}{q}-5)k'}\|\dot{\Delta}_{k'}g^{h}\|_{L^{p}} .
    \end{align*}
Consequently, for $\frac{1}{p} + \frac{2}{q} > 1$ and $p \leq 2q$,
    \begin{equation}\label{mm13}
        \begin{aligned}
            \|(R(f^{h},g^{h}))\|_{\tilde{L}_{t}^{1}(\dot{B}_{q,\infty}^{\frac{15}{q}-6})}^{\ell}
            &\lesssim \|f^{h}\|_{\tilde{L}_{t}^{\infty}(\dot{B}_{p,\infty}^{-1+\frac{3}{p}})}
                \|g^{h}\|_{\tilde{L}_{t}^{1}(\dot{B}_{p,\infty}^{1+\frac{3}{p}+ \frac{12}{q}-6})} \\
            &\lesssim \Omega^{-6|\frac{1}{2}-\frac{1}{q}|}
               \|f^{h}\|_{\tilde{L}_{t}^{\infty}(\dot{B}_{p,\infty}^{-1+\frac{3}{p}})}\|g^{h}\|_{\tilde{L}_{t}^{1}(\dot{B}_{p,\infty}^{1+\frac{3}{p}})}.
        \end{aligned}
    \end{equation}
For $2 \leq p \leq 2q$, applying H\"{o}lder's and Bernstein's inequalities again gives
    \begin{align*}
        \|\dot{\Delta}_{k}R(f^{\ell},g^{\ell})\|_{L^{q}}
        &\lesssim 2^{(\frac{6}{q}-\frac{3}{q})k}
            \Bigl\|\dot{\Delta}_{k}
            \Bigl(\sum_{k'}\tilde{\dot{\Delta}}_{k'}f^{\ell}\,
            \dot{\Delta}_{k'}g^{\ell}\Bigr)\Bigr\|_{L^{p/2}} \\
        &\lesssim 2^{(\frac{6}{q}-\frac{3}{q})k}
            \sum_{k'\geq k-2} \|\dot{\Delta}_{k'}f^{\ell}\|_{L^{q}}
            \|\dot{\Delta}_{k'}g^{\ell}\|_{L^{q}} \\
        &\lesssim 2^{(\frac{6}{q}-\frac{3}{q})k}
            \sum_{k'\geq k-2} 2^{(6-\frac{18}{q})k'}
            2^{(\frac{9}{q}-4)k'}\|\dot{\Delta}_{k'}f^{\ell}\|_{L^{q}}
            \; 2^{(-2+\frac{9}{q})k'}\|\dot{\Delta}_{k'}g^{\ell}\|_{L^{q}} .
    \end{align*}
The conditions $ \frac{1}{p} + \frac{2}{q}>1$ and $p\leq 2q$ (which can imply that $q < 3$) yields
    \begin{equation}\label{mm14}
        \|(R(f^{\ell},g^{\ell}))^{\ell}\|_{\tilde{L}_{t}^{1}(\dot{B}_{q,\infty}^{\frac{15}{q}-6})}
        \lesssim \|f^{\ell}\|_{\tilde{L}_{t}^{\infty}(\dot{B}_{q,\infty}^{\frac{9}{q}-4})}\|g^{\ell}\|_{\tilde{L}_{t}^{1}(\dot{B}_{q,\infty}^{\frac{9}{q}-2})}.
    \end{equation}
Finally, collecting all the estimates obtained for the paraproduct and remainder terms, and using the hypotheses
    $$\frac{3}{q}-\frac{3}{p}\leq\frac{9}{4q} - \frac{3}{4p} < 1 < \frac{2}{q} + \frac{1}{p},$$
summing \eqref{mm11}, \eqref{mm12}, \eqref{mm13}, \eqref{mm14}, and \eqref{mm00}, we obtain
    \begin{equation*}
        \begin{aligned}
            &\Omega^{\,3-\frac{6}{q}}
                \bigl\|(\operatorname{div}(f \otimes g))^{\ell}\bigr\|_{\tilde{L}_{t}^{1}
                (\dot{B}_{q,\infty}^{\frac{15}{q}-7})}\lesssim\|f\|_{X_{q,p}^{\Omega}} \|g\|_{X_{q,p}^{\Omega}},
        \end{aligned}
    \end{equation*}
from which, we obtain the desired results  Lemma \ref{uuulow}.

\subsection{Well-posedness}\label{wellposedness}
 Local well-posedness is classical, we omit the details. To prove global well-posedness, in light of Proposition~\ref{linearesti}, we have
     \begin{equation*}
        \begin{aligned}
            \|e^{tA_\Omega} u_{0,\Omega}\|_{_{X_{q,p}^{\Omega}}} 
           \lesssim  \Omega^{3-\frac{6}{q}}\|u_{0,\Omega}^{\ell}\|_{\dot{B}_{q,\infty}^{-7+\frac{15}{q}}}
           +\|u_{0,\Omega}^{h}\|_{\dot{B}_{p,\infty}^{-1 + \frac{3}{p}}}:= \|u_{0,\Omega}\|_{X_{0,q,p}^{\Omega}}\ll1.
        \end{aligned}
    \end{equation*}
Consequently utilizing Proposition \ref{besti}, Banach fixed point theorem ensures that the equation
$
u_{\Omega} = e^{tA_\Omega} u_{0,\Omega} + \Psi(u_{\Omega},u_{\Omega})
$
has a unique solution $u\in X_{q,p}^{\Omega}$, here we can directly choose $T=\infty$.

\section{Proof of Theorem \ref{strong convergence}}

In this section, we prove that \(u_\Omega\) converges strongly to the solution \(v\) of the non-rotating Navier--Stokes equations in the critical space as \(\Omega\to0^+\) for some $T>0$.  Let \(u_\Omega\) be the unique  local solution of \eqref{rns} with initial data $u_{0,\Omega}\in X_{0,q,p}^{\Omega}$ given by Theorem~ \ref{wp}, and let \(v\) be the solution of \eqref{INS} with initial data \(v_0\). 
Since \[
\|u_{\Omega,0}-v_0\|_{\dot B_{p,\infty}^{-1+\frac{3}{p}}} \to 0 \quad (\Omega \to 0^+),
\]  when $\delta\ll 1, \, \, \Omega\ll 1$, 
\[
\|v_{0}\|_{\dot B_{p,\infty}^{-1+\frac{3}{p}}}\lesssim \|u_{\Omega,0}\|_{\dot B_{p,\infty}^{-1+\frac{3}{p}}}+\delta^{2}\lesssim\|u_{\Omega,0}\|_{X_{0,q,p}^{\Omega}} +\delta^{2},
\]
by local well-posedness results of Theorem~\ref{wp}, and local well-posedness theory  of \eqref{INS} (see \cite{Cannone1997}), we can obtain that there are unique solutions $ u_{\Omega}$ and  $v$ for some common $T>0$ and satisfy the following uniform estimates
\begin{equation}\label{unifrom1}
\begin{aligned} 
\|u_\Omega\|_{\tilde{L}_{T}^{\infty}(\dot{B}_{p,\infty}^{-1+\frac{3}{p}})}\lesssim \delta^{2}+\|u_{0,\Omega}\|_{X_{0,q,p}^{\Omega}},\ \ 
    \|u_\Omega\|_{\tilde{L}_{T}^{1}(\dot{B}_{p,\infty}^{1+\frac{3}{p}})}\lesssim \delta
 \end{aligned}
\end{equation}
and
 \begin{equation}\label{unifrom2}
\begin{aligned}
&\|v\|_{\tilde{L}_{T}^{\infty}(\dot{B}_{p,\infty}^{-1+\frac{3}{p}})}\lesssim \delta^{2}+\|v_{0}\|_{\dot B_{p,\infty}^{-1+\frac{3}{p}}}\lesssim\delta^{2}+\|u_{0,\Omega}\|_{X_{0,q,p}^{\Omega}}, \,  \,  \,   \|v\|_{\tilde{L}_{T}^{1}(\dot{B}_{p,\infty}^{1+\frac{3}{p})}}\lesssim  \delta.
 \end{aligned}
\end{equation}
Set \(w_\Omega:=u_\Omega-v\), we obtain 
\[
\partial_t w_\Omega -\Delta w_\Omega + \Omega e_3\times u_\Omega + \PP\big( u_\Omega\cdot\nabla u_\Omega - v\cdot\nabla v\big)=0,\qquad {\rm div} w_\Omega=0,
\]
where \(\PP\) is the Leray projector. Applying Duhamel's formula,
\begin{equation}
\begin{aligned}
w_\Omega(t)= e^{t\Delta} w_\Omega(0):=e^{t\Delta}w_\Omega(0)+\mathcal{C}(t)+\mathcal{N}(t),
\end{aligned}
\end{equation}
where
\[
\mathcal{C}(t):=-\int_0^t e^{(t-s)\Delta}\mathbb{P}\bigl(\Omega e_3\times u_\Omega(s)\bigr)\,ds
\]
and
\[
\mathcal{N}(t):=-\int_0^t e^{(t-s)\Delta}\mathbb{P}\bigl[\operatorname{div}(w_\Omega\otimes u_\Omega)(s)+\operatorname{div}(v\otimes w_\Omega)(s)\bigr]\,ds.
\]
Let
\[
\|w_\Omega\|_{\mathcal{E}_T}:=\|w_\Omega\|_{\tilde{L}_T^\infty(\dot{B}_{p,\infty}^{-1+\frac{3}{p}})}+\|w_\Omega\|_{\tilde{L}_{T}^{1}(\dot{B}_{p,\infty}^{1+\frac{3}{p}})}.
\]
Our goal is to show that for some $T>0$,
\[
\lim_{\Omega\to0^+}\|w_\Omega\|_{\mathcal{E}_T}=0.
\]
For the term $e^{t\Delta}w_\Omega(0)$, we obtain
\begin{equation}\label{eslinear}
\begin{aligned}
\|e^{t\Delta}w_\Omega(0)\|_{\mathcal{E}_T}\lesssim\|w_\Omega(0)\|_{\dot{B}_{p,\infty}^{-1+3/p}}.
\end{aligned}
\end{equation}
Using \eqref{unifrom1} we have
\begin{equation}\label{escoriolis}
\begin{aligned}
\|\mathcal{C}\|_{\mathcal{E}_T}\lesssim\Omega\|u_\Omega\|_{\tilde{L}_T^1(\dot{B}_{p,\infty}^{-1+3/p})}\le\Omega T\|u_\Omega\|_{\tilde{L}_T^\infty(\dot{B}_{p,\infty}^{-1+3/p})}
\lesssim \Omega T(\delta^{2} +\|u_{0,\Omega}\|_{X_{0,q,p}^{\Omega}}).
\end{aligned}
\end{equation}
 Applying \eqref{unifrom1}, \eqref{unifrom2}, Proposition~\ref{linearesti} and Proposition \ref{besti} imply that 
\begin{equation}\label{esnonlinear}
\begin{aligned}
&\|\mathcal{N}\|_{\mathcal{E}_T}\lesssim\|\bigl[\operatorname{div}(w_\Omega\otimes u_\Omega)(s)+\operatorname{div}(v\otimes w_\Omega)(s)\bigr]\|_{L_T^1(\dot{B}_{p,\infty}^{-1+3/p})}\\
&\lesssim \|w_\Omega\|_{\tilde{L}_{T}^{2}(\dot{B}_{p,\infty}^{\frac{3}{p}})}\|u_\Omega\|_{\tilde{L}_{T}^{2}(\dot{B}_{p,\infty}^{\frac{3}{p}})}+ \|w_\Omega\|_{\tilde{L}_{T}^{2}(\dot{B}_{p,\infty}^{\frac{3}{p}})}\|v\|_{\tilde{L}_{T}^{2}(\dot{B}_{p,\infty}^{\frac{3}{p}})}\\
&\lesssim \|w_\Omega\|_{\mathcal{E}_T}(\|u_\Omega\|_{\tilde{L}_{T}^{1}(\dot{B}_{p,\infty}^{1+\frac{3}{p}})}^{\frac{1}{2}}\|u_\Omega\|_{\tilde{L}_{T}^{\infty}(\dot{B}_{p,\infty}^{-1+\frac{3}{p}})}^{\frac{1}{2}}+\|v\|_{\tilde{L}_{T}^{1}(\dot{B}_{p,\infty}^{1+\frac{3}{p}})}^{\frac{1}{2}}
\|v\|_{\tilde{L}_{T}^{\infty}(\dot{B}_{p,\infty}^{-1+\frac{3}{p}})}^{\frac{1}{2}})
\\
&\lesssim(\delta^{3}+\delta\|u_{0,\Omega}\|_{X_{0,q,p}^{\Omega}})^{\frac{1}{2}}\|w_\Omega\|_{\mathcal{E}_T}.
\end{aligned}
\end{equation}
Finally, combining \eqref{eslinear},  \eqref{escoriolis} and \eqref{esnonlinear}, we arrive at
\[
\|w_\Omega\|_{\mathcal{E}_T}\le K\|w(0)\|_{\dot{B}_{p,\infty}^{-1+3/p}}+ K\Omega T(\delta^{2}+\|u_{0,\Omega}\|_{X_{0,q,p}^{\Omega}})+K(\delta^{3}+\delta\|u_{0,\Omega}\|_{X_{0,q,p}^{\Omega}})^{\frac{1}{2}}\|w_\Omega\|_{\mathcal{E}_T}
\]
for some constant $K>0$. Since $\delta\ll 1$, we can get 
\[
\|w_\Omega\|_{\mathcal{E}_T}\le 2K\|w_\Omega(0)\|_{\dot{B}_{p,\infty}^{-1+3/p}}+2K\Omega T(\delta^{2} +\|u_{0,\Omega}\|_{X_{0,q,p}^{\Omega}}).
\]
Since $\|w_\Omega(0)\|_{\dot B_{p,\infty}^{-1+\frac{3}{p}}}$ converges to 0 as $\Omega \to 0^+$, hence for some $T>0$, we obtain
\[
\lim_{\Omega\to0^+}\|w_\Omega\|_{\mathcal{E}_T}=0.
\]
That is, $u_\Omega$ converges strongly to $v$ in the space
 $$
\tilde{L}_{T}^{\infty}(\dot{B}_{p,\infty}^{-1+\frac{3}{p}})\cap
\tilde{L}_{T}^{1}(\dot{B}_{p,\infty}^{1+\frac{3}{p}}),$$
 which completes the proof of the vanishing rotational limit.
\section{Appendix}
In this appendix, we would present some basic functional analysis tools.
Let $\dot\Delta_j$ for $j\in\Z$ be the Fourier multiplier defined in Subsection 1.2. the homogeneous Besov space is defined as:

\hspace*{\fill}
\begin{definition}\label{defn2.1}
For $s\in \mathbb{R}$ and $1\leq p,r\leq \infty$, the homogeneous Besov spaces $\dot{B}^s_{p,r}$ are defined by
$$\dot{B}^s_{p,r}:=\Big\{f\in \mathcal{S}'_{0}:\|f\|_{\dot{B}^s_{p,r}}<\infty  \Big\},$$
where
\begin{equation*}
\|f\|_{\dot{B}^s_{p,r}}:=\Big(\sum_{q\in\mathbb{Z}}(2^{qs}\|\dot{\Delta}_qf\|_{L^{p}})^{r}\Big)^{1/r}
\end{equation*}
with the usual convention if $r=\infty$.
\end{definition}

We often use the following classical properties of Besov spaces (see \cite{BCD}):

$\bullet$ \ \emph{Scaling invariance:} For any $\sigma\in \mathbb{R}$ and $(p,r)\in
[1,\infty ]^{2}$, there exists a constant $C=C(\sigma,p,r,d)$ such that for all $\lambda >0$ and $f\in \dot{B}_{p,r}^{\sigma}$, we have
$$
C^{-1}\lambda ^{\sigma-\frac {d}{p}}\|f\|_{\dot{B}_{p,r}^{\sigma}}
\leq \|f(\lambda \,\cdot)\|_{\dot{B}_{p,r}^{\sigma}}\leq C\lambda ^{\sigma-\frac {d}{p}}\|f\|_{\dot{B}_{p,r}^{\sigma}}.
$$

$\bullet$ \ \emph{Completeness:} $\dot{B}^{\sigma}_{p,r}$ is a Banach space whenever $
\sigma<\frac{d}{p}$ or $\sigma\leq \frac{d}{p}$ and $r=1$.

$\bullet$ \ \emph{Interpolation:} The following inequality is satisfied for $1\leq p,r_{1},r_{2}, r\leq \infty, \sigma_{1}\neq \sigma_{2}$ and $\theta \in (0,1)$:
$$\|f\|_{\dot{B}_{p,r}^{\theta \sigma_{1}+(1-\theta )\sigma_{2}}}\lesssim \|f\| _{\dot{B}_{p,r_{1}}^{\sigma_{1}}}^{\theta} \|f\|_{\dot{B}_{p,r_2}^{\sigma_{2}}}^{1-\theta }$$
with $\frac{1}{r}=\frac{\theta}{r_{1}}+\frac{1-\theta}{r_{2}}$.

$\bullet$ \ \emph{Action of Fourier multipliers:} if $F$ is a smooth homogeneous of
degree $m$ function on $\mathbb{R}^{d}\backslash \{0\}$, then
$$F(D):\dot{B}_{p,r}^{\sigma}\rightarrow \dot{B}_{p,r}^{\sigma-m}.$$


The embedding properties will be used several times throughout the paper.
\begin{proposition} \label{Prop2.1}
\begin{itemize}
\item For any $p\in[1,\infty]$, we have the continuous embedding
$\dot{B}^{0}_{p,1}\hookrightarrow L^{p}\hookrightarrow \dot{B}^{0}_{p,\infty}$.
\item If $\sigma\in\mathbb{R}$, $1\leq p_{1}\leq p_{2}\leq\infty$ and $1\leq r_{1}\leq r_{2}\leq\infty$, then
$\dot{B}^{\sigma}_{p_{1},r_{1}}\hookrightarrow
\dot{B}^{\sigma-d\,(\frac{1}{p_{1}}-\frac{1}{p_{2}})}_{p_{2},r_{2}}$.
\item The space $\dot{B}^{d/p}_{p,1}$ is continuously embedded in the set of
bounded continuous functions (vanishing at infinity if $p<\infty$).
\end{itemize}
\end{proposition}
In addition, we also recall the classical \emph{Bernstein inequality}
\begin{equation}\label{Eq:2.6}
\|D^{k}f\|_{L^{b}}
\leq C^{1+k} \lambda^{k+d(\frac{1}{a}-\frac{1}{b})}\|f\|_{L^{a}}
\end{equation}
that holds for all function $f$ such that $\mathrm{Supp}\,\mathcal{F}f\subset\left\{\xi\in \mathbb{R}^{d}: |\xi|\leq R\lambda\right\}$ for some $R>0$
and $\lambda>0$, if $k\in\mathbb{N}$ and $1\leq a\leq b\leq\infty$.

More generally, if we assume $f$ to satisfy $\mathrm{Supp}\,\mathcal{F}f\subset\{\xi\in \mathbb{R}^{d}:
R_{1}\lambda\leq|\xi|\leq R_{2}\lambda\}$ for some $0<R_{1}<R_{2}$ and $\lambda>0$, then for any smooth
homogeneous of degree $m$ function $A$ on $\mathbb{R}^d\setminus\{0\}$ and $1\leq a\leq\infty$, we have
(see e.g. Lemma 2.2 in \cite{BCD}):
\begin{equation}\label{Eq:2.7}
\|A(D)f\|_{L^{a}}\approx\lambda^{m}\|f\|_{L^{a}}.
\end{equation}
An obvious  consequence of (\ref{Eq:2.6}) and (\ref{Eq:2.7}) is that
$\|D^{k}f\|_{\dot{B}^{s}_{p, r}}\thickapprox \|f\|_{\dot{B}^{s+k}_{p, r}}$ for all $k\in\mathbb{N}$.

Moreover, a class of mixed space-time Besov spaces are also used when studying the evolution PDEs, which were firstly
proposed by \cite{BCD}.
\begin{definition}\label{defn2.2}
For $T>0,s\in \mathbb{R}, 1\leq r,\theta \leq \infty$, the homogeneous Chemin-Lerner spaces $\tilde{L}^\theta_{T}(\dot{B}^s_{p,r})$ are defined by
$$\tilde{L}^\theta_{T}(\dot{B}^s_{p,r}):=\Big\{f\in L^\theta(0,T;\mathcal{S}'_{0}) :\|f\|_{\tilde{L}^\theta_{T}(\dot{B}^s_{p,r})}<\infty  \Big\}, $$
where
$$\|f\|_{\tilde{L}^\theta_{T}(\dot{B}^s_{p,r})}:=\Big(\sum_{q\in \mathbb{Z}}(2^{qs}\|\dot{\Delta}_qf\|_{L^\theta_{T}(L^{p})})^{r}\Big)^{1/r} $$
with the usual convention if  $r=\infty$.
\end{definition}

The Chemin-Lerner space $\widetilde{L}^{\theta}_{T}(\dot{B}^{s}_{p,r})$ may be linked with the standard spaces $L_{T}^{\theta}(\dot{B}^{s}_{p,r})$ by means of Minkowski's inequality.
\begin{remark}\label{Rem2.1}
It holds that
\[
\|f\|_{\widetilde{L}^{\theta}_{T}(\dot{B}^{s}_{p,r})} \leq \|f\|_{L^{\theta}_{T}(\dot{B}^{s}_{p,r})}
\quad \text{if } r \geq \theta; \qquad
\|f\|_{\widetilde{L}^{\theta}_{T}(\dot{B}^{s}_{p,r})} \geq \|f\|_{L^{\theta}_{T}(\dot{B}^{s}_{p,r})}
\quad \text{if } r \leq \theta.
\]
\end{remark}

\bigbreak\bigbreak\noindent
\noindent 
\textbf{Acknowledgments.~} Z. Guo is the recipient of an Australian Research Council Future Fellowship (project number FT230100588) funded by the Australian Government and is also supported by ARC DP260100485. M. Yang is supported by the National Natural Science Foundation of China (project number 12161041) and the Jiangxi Province Natural Science Foundation (project number 20252BAC250004).

\noindent
\textbf{Conflict of interest.} The authors do not have any possible conflicts of interest.

\noindent
\textbf{Data availability statement.}
 Data sharing is not applicable to this article, as no data sets were generated or analyzed during the current study.




\end{document}